\documentclass[preprint,12pt]{elsarticle}
\usepackage{geometry} 
\usepackage{amssymb}
 \usepackage{amsthm}
\usepackage{amsmath}
\usepackage{graphicx}
\usepackage{mathrsfs}
\usepackage{srcltx}
\newtheorem{theorem}{Theorem}
\newtheorem{definition}{Definition}

\journal{Journal of Mathematical Analysis and Applications}

\begin{document}

\begin{frontmatter}



\title{The Bohr radius and the Hadamard convolution operator}


\author{Khasyanov R.Sh.\fnref{label2}}
\ead{st070255@student.spbu.ru}
\affiliation[label2]{organization={Saint Petersburg State University},
            addressline={Universitetskii prosp., 28D}, 
            city={Saint Petersburg},
            postcode={198504}, 
            country={Russian Federation}}

\begin{abstract}
The concept of the Bohr radius of a pair of operators is introduced. In terms of the convolution function, a general formula for calculating the Bohr radius of the Hadamard convolution type operator with a fixed initial coefficient  is obtained. We apply this formula to the problems of the Bohr radius of the operators of differentiation and integration.  Using the concept of the Bohr radius of a pair of operators, we generalize the theorem of B.Bhowmik and N.Das on the comparison of majorant series of subordinate functions.

\end{abstract}




\end{frontmatter}


\section{Introduction}
In 1914, H. Bohr, studying Dirichlet series, noticed \cite{Bohr} the following interesting fact in complex analysis, which is now called the Bohr phenomenon:

\newtheorem*{A}{Theorem A}
\begin{A}
	Let $f(z)=\sum_{n\ge 0}a_nz^n$ and $\|f\|_{\infty}:=\sup_{z\in \mathbb{D}}|f(z)|\le1$ in the unit disk $\mathbb{D}=\{|z|<1\}$. Then
	$$\sum_{n\ge 0}^{} |a_n| r^n\le1, \quad
	0\le r\le 1/3.$$
	The constant $1/3$ is sharp.
\end{A} 
The Bohr inequality was proved by H. Bohr in 1914 only for $r\le 1/6$, the constant $1/3$ was obtained independently in the same year by M.  Riesz, I. Schur and F. Wiener (for different proofs see e.g. appendix in \cite{Dixon05}).
This theorem can also be proved using the Schwarz-Pick inequality at zero: $|a_1|\le 1-|a_0|^2.$ The method of symmetrization of the analytic function $$g_n(z)=\frac{f(z)+f(e^{\frac{2\pi i}{n}}z)+...+f(e^{\frac{2\pi i(n-1)}{n}}z)}{n},$$
gives us $|a_n|\le 1-|a_0|^2.$ Hence 
$$\sum_{n\ge 0}|a_n|r^n\le |a_0|+(1-|a_0|^2)\frac{r}{1-r}.$$
If $r\le 1/3$  then the last expression is not greater than $1-\frac{(1-|a_0|)^2}{2}<1.$ The Schwarz-Pick inequality turns into the equality only for conformal automorphisms of the disk $$f_a(z)=\frac{z-a}{1-\bar{a}z}, \quad 0\le |a|<1.$$ The  functional family $f_a, \:a\rightarrow 1,$ shows that the constant $1/3$ is the best possible.

Bohr's theorem is equivalent to the inequality known as the Bohr inequality:

$$
\sum_{n\ge0}|a_n|r^n\le \|f\|_\infty, \quad  0\le r\le 1/3.
$$   

An active study of various modifications and generalizations of the Bohr inequality began in the middle of 1990s since P. Dixon, using the Bohr inequality, solved a long-standing problem about the characterization of Banach algebras \cite{Dixon05}. Part of the subsequent research in this area
is directed towards extending the Bohr phenomenon in multidimensional framework and in more
abstract settings (see e.g. \cite{Aizen}, \cite{Boas}, \cite{BoasKhav}, \cite{PPS}). Paulsen and Singh extended Bohr’s inequality
to Banach algebras in \cite{PaSi}. In 2018 B. Bhowmik and N. Das applied the Bohr inequality to the question of comparing majorant series of subordinating functions \cite{BhowDas}. For studies related to the Bohr inequality, see, for example,  \cite{Aizen}-\cite{Dixon05}, \cite{Kayump}-\cite{Ricci}.

\vspace{3mm}
\section{Definitions and notations}
We introduce the concept of the Bohr radius of a pair of operators, in terms of which many well-known results related to the Bohr inequality can be formulated. Let us first recall the definition of the majorant series of analytic function:

\begin{definition}
	
	Let $f(z)=\sum_{n\ge 0}a_nz^n, \: z\in \mathbb{D}.$ The functional $$M_rf:=\sum_{n\ge 0}|a_n|r^n$$ is called the majorant series or the Bohr sum of the function $f.$

\end{definition}

Let $D$ be an open disk centered at zero or an open interval centered at zero. Denote
 $\mathcal{H}(D)$ the set of all functions of the form $f(z)=\sum_{n\ge 0}a_nz^n$ that converge in $D$.
$$\mathcal{H}_m(D):=\{f\in \mathcal{H}(D): \: f(0)=f^{\prime}(0)=...=f^{(m-1)}(0)=0\}.$$

\begin{definition}
	
	Let $m,t,s_1,s_2\ge 0$ and $$T_1:\mathcal{H}_m(t\mathbb{D})\rightarrow \mathcal{H}(\mathbb{D}), \hspace{5mm} T_2: \mathcal{H}_m(-s_1,s_1) \rightarrow \mathcal{H}(-s_2,s_2)$$ be linear operators. We call the maximal $R\in [0,\infty)$, for which

	\begin{equation} \label{def}
		\|T_1f\|_{\infty}\le 1 \Longrightarrow |T_2M_rf|\le 1, \quad 0\le  r\le R, \end{equation}
	the Bohr radius of a pair of $T_1$ and $T_2$ and denote $R_{T_1\rightarrow T_2}.$ If $T_1=T_2,$ we write $R_{T_1}$. 

\end{definition}

As in the case of the classical Bohr inequality, the condition \eqref{def} is equivalent to the inequality
$$|T_2M_rf|\le \|T_1f\|_{\infty}.$$

Denote by $id_m$ the identity operator defined in the space $\mathcal{H}_m(\mathbb{D}).$ The notations $R_{id\to T}$ and $R_{T\to id}$ mean that the operator $id$ is the identity operator of the space on which the operator $T$ is defined. It follows from Bohr's theorem that $R_{id_0}=1/3.$ E. Bombieri in 1962 proved \cite{Bomb} that $R_{id_1}=1/\sqrt{2}.$ The problem of finding $R_{id_m}$ is open. It was discussed in \cite{PPS} and \cite{Bombb}. Note that according to the definition above, the Bohr radius can be greater than one. For example, for the operator $Tf(z)=f(4z)$ the Bohr radius $R_{T\rightarrow id}$ is $4/3.$

\begin{definition}

	The function $$m_{T_1\rightarrow T_2}(r):=\sup_{f: \|T_1f\|_{\infty}\le 1}\dfrac{|T_2M_rf|}{\|T_1f\|_{\infty}}$$
	will be called the Bohr-Bombieri function of a pair $T_1$ and $T_2.$  If the operators $T_1$ and $T_2$ coincide, we write $m_{T_1}(r).$
	
\end{definition}

\textbf{Examples.}

\vspace{2mm}
1. It follows from Bohr's theorem that $m_{id_0}(r)=1, \:0\le r\le 1/3.$ In 1962, E. Bombieri and D.~Ricci proved that for $r\in [1/3, 1/\sqrt{2}],$
$$m_{id_0}(r)=\dfrac{3-\sqrt{8(1-r^2)}}{r}.$$
For $r\in [1/\sqrt{2},1)$, the problem of calculating the Bohr-Bombieri function for the operator $id_0$ remains open and is directly related to the problem of finding $R_{id_m}, \:m\ge 2.$ Profound results related to this problem were obtained by E. Bombieri and J. Bourgain in 2004 (see \cite{Bombb}).

\vspace{3mm}
2. Denote by $\mathcal{C}f(z):=\sum_{n\ge 0}\big(\frac{1}{n+1}\sum_{k= 0}^{n}a_k\big)z^n$ the Cesaro operator. I.~ Kayumov, S. Ponnusamy and D. Khammatova proved \cite{KayumpCesaro} that 
$$m_{id\rightarrow\mathcal{C}}(r)\le \dfrac{1}{r}\log{\dfrac{1}{1-r}}, \quad r\le 0.5335...$$
\begin{definition}
	
	Consider the functions of the form $f(z)=\sum_{n\ge m}c_nz^n.$ Fix the modulus of the initial coefficient $a:=|c_m |$. 
	We call the maximal number $r=|z|,$ at which for all functions of the form $f(z)=\sum_{n\ge m}c_mz^m, \: |c_m|=a,$  the condition (1) holds,
	the Bohr radius of a pair  $T_1$ and $T_2$ with an initial coefficient $a$. Similarly, we define the Bohr-Bombieri function of a pair of operators with an initial coefficient. 
	 Denote, respectively, $$R_{T_1\rightarrow T_2}(a), \hspace{5mm} m_{T_1\rightarrow T_2}(r,a).$$

\end{definition}

For instance  $$R_{id_0}(a)=\dfrac{1}{1+2a}, \quad 1/2<a\le 1.$$
This result was proved by E.~ Bombieri and D.~Ricci in \cite{Bomb}.

\vspace{3mm}
\section{Bohr radius and the Hadamard convolution operator}

\begin{definition}
	Let $f(z)=\sum_{n \ge 0}a_nz^n$ and $h(z)=\sum_{n \ge 0}c_nz^n.$ Operator
	$$A_hf(z)=(h\ast f)(z):=\sum_{n\ge 0}c_na_nz^n$$
	is called the Hadamard convolution operator (see, for example, \cite{Rusch}). We will call the function $h$ the convolution function.
\end{definition}

\vspace{2mm}

\vspace{2mm}
In \cite{ABS} R. Ali, R. Barnard and A. Solynin considered the problem of calculating the Bohr radius for even functions and, as a corollary, obtained the value of the Bohr radius in the problem of the estimating the modulus of an alternating series of analytic function through its uniform norm in the disk. In our terms, their results can be formulated as follows:
$$R_{A_{\frac{1}{1-z^2}}}=R_{id \rightarrow A_{\frac{1}{1+z}}}=1/\sqrt{3}.$$

\vspace{2mm}
I. Kayumov, S. Ponnusamy and D. Khammatova obtained the following general result related to the Bohr radius and the convolution operator:

\newtheorem*{B}{Theorem B (I. Kayumov, S. Ponnusamy and D. Khammatova, \cite{KayumpCesaro2})}
\begin{B}
	Let $\{\phi_k(r)\}_{k=0}^{\infty}$ be a sequence of non-negative and continuous functions in $[0,1)$ such that the series $\sum_{k=0} ^{\infty}\phi_k(r)$ converges locally uniformly with respect to $r\in [0,1).$
	Let $f(z)=\sum_{n\ge 0}a_nz^n,$ $p\in (0,2]$ and
	$$
	\phi_0(r)>\dfrac{2}{p}\sum_{n\ge1}\phi_n(r), \quad r<R,
	$$
	where $R$ is the minimal positive root of the equation
	$$
	\phi_0(x)=\dfrac{2}{p}\sum_{n\ge1}\phi_n(x).$$
	Then the following inequality is true:
	$$|a_0|^p\phi_0(r)+\sum_{n\ge1}|a_n|\phi_n(r)\le \phi_0(r), \quad r\le R.$$
	If $$
	\phi_0(r)<\dfrac{2}{p}\sum_{n\ge1}\phi_n(r) $$ on some interval $(R, R+\varepsilon)$, then the number $R$ is the best possible.
\end{B}

In particular, using the last theorem, the authors found the value of the Bohr radius of the identity operator and the convolution operator with
hypergeometric Gaussian function. Let $$F(z)=F(a,b,c,z):=\sum_{n\ge 0}\gamma_nz^n, \quad \gamma_n:=\dfrac{(a)_n(b)_n}{(c)_n(1)_n},$$
where $a,b,c>-1,$ such that $\gamma_n\ge 0,$ $\:\:(a)_n:=a(a+1)...(a+n-1), \: (a)_0=1.$ Then
$R_{id\rightarrow A_F}$ is the minimal positive root of the equation $|F(a,b,c,x)-1|=1/2.$

\vspace{3mm}
\section{The main theorem}

Let $m\ge 0, \:l\in \mathbb{Z},\: m+l\ge 0$ and $S_{m,l}$ be the shift operator on space $\mathcal{H}_m(D)$ namely $S_{m,l}f(z):=z^{l}f(z), \: f\in \mathcal{H}_m(D).$ Let $h(z)=\sum_{n\ge m} c_nz^n.$   We will consider operators of form \begin{equation} \label{Amh} A^{m,l}_{h}f:=S_{m,l}(h \ast f), \quad A^m_h:=A^{m,0}_h, \quad A_h:=A^0_h.
\end{equation} In particular, examples of such operators are differentiation operators on $\mathcal{H}_m(\mathbb{D}):$ $$\partial^mf(z):=\dfrac{d^m}{dz^m}f(z)=m!\cdot S_{m,-m}\Big(\dfrac{z^m}{(1-z)^{m+1}} \ast f(z)\Big)$$
and the integration operator
\begin{equation}\label{int}
\int_0^zf(\zeta)d\zeta=S_{0,1} \Big(\dfrac{-\log(1-z)}{z}\ast f(z)\Big).	
\end{equation}
Let us consider the class $\mathcal{K}$ of convex, univalent analytic functions such that $f(0)=0$ and $f^{\prime}(0)=1.$ We use the notation $\overline{co} \mathcal{K}$ for the closed convex hull of $\mathcal{K}.$
	\begin{theorem}\label{Th:1}
	Let $h_1(z)=\sum_{n\ge m}c_nz^n, \: c_n>0, \: h_2(z)=\sum_{n\ge m}d_nz^n, \:d_n\ge 0$ and assume that there exists a sequence $b_n$ such that $\: b_n^2 = d_{n+m}, \: n\ge 1$ and the function   $\widetilde{h_2}(z)=\sum_{n\ge 1}b_nz^n $  belongs to  $\overline{co}\mathcal{K}.$ If $r\le \inf_{n\ge m+1}\dfrac{c_{n}}{c_{n+1}}$ and $a=|a_m|>r$ then 
	$$m_{id \rightarrow A^{m,l}_{h_1\ast h_2}}(r,a)=r^{m+l}c_md_ma+(1/a-a)a^{-m}r^l((h_1\ast h_2)(ar)-c_md_m(ar)^m).$$
\end{theorem}

In the proof of the above theorem, we use the methods of I. Kayumov and S. Ponnusamy (see e.g. \cite{KayumpLacun}). In particular, we use one of Goluzin's theorems on the majorization of subordinate functions. We also use theorem of  J. Sok\'ol \cite{Sok} which is a generalization of the theorem of S.~Ruscheweyh, J. Stankiewicz on the subordination under convex univalent functions \cite{RS}. Let us first formulate the definition of subordinate functions:

\begin{definition}
	We say that the function $f(z)$ is subordinate to the function $g(z)$ and write $f\prec g$ if there exists an analytic function $\omega(z)$ such that $|\omega(z)|\le 1, z\in \mathbb{D},$ $\omega(0)=0$ and
	$$f(z)=g(\omega(z)), \quad z\in \mathbb{D}.$$
	
\end{definition} 

\newtheorem*{C}{Theorem C (G. Goluzin, \cite{Gol})}
\begin{C}
	Let $f(z)=\sum_{n\ge 0}a_nz^n, \: g(z)=\sum_{n\ge 0}b_nz^n$  and  $f\prec g$ in $\mathbb{D}.$ Then for any non-increasing sequence $\lambda_n\ge 0,$
	$$\sum_{n\ge 1}\lambda_n|a_n|^2\le  \sum_{n\ge 1}\lambda_n|b_n|^2.$$
\end{C}

\newtheorem*{D}{Theorem D (J. Sok\'ol, \cite{Sok})}
\begin{D}
Let $F$ and $G$ be in $\overline{co}\mathcal{K}.$ If $f\prec F$ and $g\prec G,$ then $f\ast g\prec F\ast G.$

\end{D}
\newtheorem*{L1}{Lemma}
\begin{L1}
	Let $m\ge 0, \:h_1(z)=\sum_{n\ge m}c_nz^n, \:c_n> 0,\: n\ge m+1,$ \: $0\le x\le \inf_{n\ge m+1} \dfrac{c_{n}}{c_{n+1}}$. Let $h_2(z)=\sum_{n\ge m}d_nz^n$ and the sequence $d_n\ge0, \:n\ge m$ satisfies the condition of Theorem 1, $f(z)=\sum_{n\ge m}a_nz^n, \: \|f\|_{\infty}\le 1, \: a=|a_m|.$ Then
	$$\sum_{n\ge m+1}c_nd_n|a_n|^2x^n\le (1/a-a)^2a^{-2m}((h_1\ast h_2)(a^2x)-c_md_m(a^2x)^m).$$ 
\end{L1}

\begin{proof}[Proof.]
	By Schwarz' lemma $\Big|\dfrac{f(z)}{z^m}\Big|\le 1, z\in \mathbb{D},$	hence
	$\dfrac{f(z)}{z^m}\prec \dfrac{z+a}{1+az}.$
	Using Theorem D, we get
	$$\dfrac{f(z)}{z^m}\ast \widetilde{h_2}(z)\prec \dfrac{z+a}{1+az}\ast \widetilde{h_2}(z).$$
	The last expression is equivalent to
	$$\sum_{n\ge 1}b_na_{n+m}z^n\prec \sum_{n\ge 1}b_n (a-1/a)(-az)^n.$$
	Applying the Goluzin theorem to the functions $\dfrac{f(z)}{z^m}\ast \widetilde{h_2}(z), \hspace{3mm} \dfrac{z+a}{1+az}\ast \widetilde{h_2}(z)$ and sequence $\lambda_n=c_{n+m}x^n, n\ge 1,$  that is non-increasing by the condition of the lemma, we get the desired inequality.
\end{proof}

\vspace{3mm}

\begin{proof}[Proof of Theorem 1.] Let us estimate the sum $A_{h_1\ast h_2}^{m,l}M_rf,$ using the Cauchy-Bunyakovsky inequality and the proved lemma:
	\begin{multline}\label{chain}
		A_{h_1\ast h_2}^{m,l}M_rf=\sum_{n\ge m}c_nd_n|a_n|r^{n+l} \\ \le
		 c_md_mar^{m+l}+r^l\Big(\sum_{n\ge m+1}c_nd_n|a_n|^2\rho^{2n}\Big)^{1/2}\Big({\sum_{n\ge m+1}c_nd_n(r/\rho)^{2n}}\Big)^{1/2}  \\  \\
		 \le 
	 c_md_mar^{m+l}+(1/a-a)a^{-m}r^l({(h_1\ast h_2)((a\rho)^2)-c_md_m(a\rho)^{2m}})^{1/2}\cdot \\ \\ \cdot ({(h_1\ast h_2)((r/\rho)^2)-c_md_m(r/\rho)^{2m}})^{1/2}. 
	\end{multline}
	We choose $\rho$ so that the inequalities are sharp. Obviously, the inequality proved in the lemma is sharp if and only if
	$$f(z)=z^m\dfrac{z+a}{1+az}=z^m\Big(a+(a-1/a)\sum_{n\ge 1}(-az)^n\Big).$$
	The applied Cauchy-Bunyakovsky inequality is sharp if
	$$|a_n|\rho ^n=C\Big(\dfrac{r}{\rho}\Big)^n, \quad n\ge m+1,$$
	where $C$ is a constant independent of $n$. Substituting the coefficients of the function $f(z)$ into the last formula, we conclude that $\rho=\sqrt{r/a}$. This imposes the restriction $a> r$. Substituting $\rho=\sqrt{r/a}$ into  \eqref{chain}, we obtain the required formula.
\end{proof}

\vspace{2mm}

\section{Corollaries}

In what follows we will write $h(z)$ instead of $h_1(z)$ and put $h_2(z)=\dfrac{z^m}{1-z}.$ Note that $\widetilde{h_2}(z)=\dfrac{z}{1-z}\in \mathcal{K}.$

1. Consider the problem of calculating $R_{id_0}(a).$ As already noted, for the case $1/2< a \le 1$ this problem was solved by E. Bombieri and D.~Ricci in 1962. Let us show that their result is a special case of our theorem. Notice, that
$$id_0f=h\ast f, \quad h(z)=\dfrac{1}{1-z}.$$
Then Theorem 1 immediately implies that for all $r\le 1$ and $a\in (r,1],$
$$m_{id_0}(r,a)=a+r\cdot \dfrac{1-a^2}{1-ar}.$$
Solving the equation $m_{id_0}(r,a)=1,$ we get $R_{id_0}(a)=\dfrac{1}{1+2a}$ for all $a$ such that $a> R_{id_0}(a),$ i.e. for $a\in (1/2,1].$ 

\vspace{5mm}
2. In \cite{GolDeriv}, G. Goluzin proved that $R_{id_1\rightarrow \partial}=1-\sqrt{2/3}.$ Considering the function $f(z)=z^m$ , it becomes obvious that $R_{id_m\rightarrow \partial^m}=0, \: m\ge 2.$ However, if we divide the differentiation operator by $m!$, then the Bohr radius will always be greater than zero. This fact can be proved using only the Schwarz-Pick inequality, as in the proof of the classical Bohr theorem:

	\begin{theorem}\label{Th:2}
	$$R_{id_m\rightarrow \partial^m/m!}=1-\sqrt[m+1]{2/3}, \quad m\ge 0.$$
\end{theorem}

\begin{proof}[Proof.]
	
	We denote $a=|a_m|.$
	Since $\|f\|_{\infty}\le 1$ then $|a_n|\le 1-a^2, \: n\ge 1.$ Let us estimate the sum $\partial^mM_rf/m!:$
	$$\partial^mM_rf/m!=\dfrac{1}{m!}\sum_{n\ge m}n(n-1)...(n-m+1) |a_n|r^{n-m}\le a+\Big(\dfrac{1}{(1-r)^{m+1}}-1\Big)(1-a^2).$$ 
	The last expression is not greater than one if $r\le 1-\sqrt[m+1]{\dfrac{1+a}{2+a}},$ so $R_{id_m\rightarrow \partial^m/m! }\ge 1-\sqrt[m+1]{2/3}.$ The inverse estimate is obtained by considering the family of functions $z^m\dfrac{z-a}{1-az}, \: a\rightarrow 1.$
\end{proof}

\vspace{3mm}
Note that the proof of Theorem 2 implies that $$R_{id_m\rightarrow \partial^m/m!}(a)\ge 1-\sqrt[m+1]{\dfrac{1+a}{2+ a}}$$
for all $a\in [0,1).$ Let us show that using Theorem 1  we can obtain the sharp value of $R_{id_m\rightarrow \partial^m/m!}(a)$ for some $a.$

	\begin{theorem}\label{Th:3}
	If $a^2>1-\sqrt[m+1]{\frac{1+a}{1+2a}}, \: 0<a<1$ then 
	$$R_{id_m\rightarrow \partial^m/m!}(a)=\dfrac{1}{a}\Big(1-\sqrt[m+1]{\dfrac{1+a}{1+2a}}\Big), \quad m\ge 0.$$
\end{theorem}

\begin{proof}[Proof.]
	
	As noted earlier,
	$$\dfrac{\partial^m}{m!}f=S_{m,-m}(h\ast f), \quad h(z)=\dfrac{z^m}{(1-z) ^{m+1}}.$$
	Therefore, Theorem 1 implies that for $r\le \dfrac{2}{2+m}$ and $a>r,$
	$$m_{id\rightarrow \partial^m/m!}(r,a)=a+(1/a-a)\Big(\dfrac{1}{(1-ar)^{m+1}}-1 \Big).$$
	Equating the last expression to one, we obtain $r(a)=\dfrac{1}{a}\Big(1-\sqrt[m+1]{\dfrac{1+a}{1+2a}}\Big).$ The inequality $r(a)\le \dfrac{2}{2+m}, \:m\ge 0,$ is easily verified.
\end{proof}

\vspace{3mm}
3. Let us now consider the problem of estimating the majorant series via the norm of its derivative. The specifics of the problems under consideration lies in the fact that often the solution  is the root of the transcendental equation. In this regard, sometimes the answer is written out in terms of the Lambert function, which is defined as the inverse function to $g(w)=we^w$ and is denoted by $W(x).$ This function is defined for $x\in [-1/e, \infty).$ For positive $x$, the Lambert function is uniquely defined, but for $x\in [-1/e, 0)$ the function $W(x)$ has two values. We choose the branch of the function $W$ for which $W(x)\ge -1.$
	\begin{theorem}\label{Th:4}
	$$0.872664...\le R_{\partial \rightarrow id_1} \le 0.883677... $$
\end{theorem} 

\begin{proof}[Proof.]
	Let us prove the lower bound for $R_{\partial \rightarrow id_1}$. Let $\|f^{\prime}\|_{\infty}\le 1,$ then
	\begin{multline*}
		\sum_{n\ge 1}|a_n|r^n \le \Big(\sum_{n\ge 1}n^2|a_n|^2\Big)^{1/2}\Big(\sum_{n\ge 1}\dfrac{r^{2n}}{n^2}\Big)^{1/2} =\|f^{\prime}\|_2(Li_2(r^2))^{1/2}\le \|f^{\prime}\|_{\infty}(Li_2(r^2))^{1/2}\le (Li_2(r^2))^{1/2},
	\end{multline*}
	where $Li_2(x)=\sum_{n\ge 1}\dfrac{x^n}{n^2}$ is a polylogarithmic function. $Li_2(r^2)=1$ for $r=0.872664...$, and thus the lower bound is proved.
	
	\vspace{2mm}
	Let us now prove the upper bound for $R_{\partial \rightarrow id_1}$. To do this, consider the function $f(z)=\sum_{n\ge 1}a_nz^n$:
	$$f(z)=\int_{0}^{z}\dfrac{\zeta-a}{1-a\zeta}d\zeta,\quad 0<a<1.$$
	It is obvious that $\|f^{\prime}\|_{\infty}\le 1$. We calculate $M_rf:$
	$$M_rf=ar+\dfrac{a^2-1}{a^2}(\log{(1-ar)}+ar).$$ Equating the last expression to one, we get the following solution:
	$$r(a)=\dfrac{1}{a}\Big(1+\dfrac{a^2-1}{2a^2-1}W\Big(\dfrac{1-2a^2}{ a^2-1}e^{-1}\Big)\Big),$$
	where $W(x) $ is the Lambert function. 
	$$\min_{0<a<1}r(a)=r(0.812308...)=0.883677...$$ hence $R_{\partial \rightarrow id_1}\le 0.883677....$
\end{proof}

\vspace{4mm}
	\begin{theorem}\label{Th:5}
	Let	 $a\in (0.892643..., 1].$ Then
	$$R_{\partial\rightarrow id_1}(a)=\dfrac{1}{a}\Big(1+\dfrac{a^2-1}{2a^2-1}W\Big(\dfrac{1-2a^2}{a^2-1}e^{-1}\Big)\Big).$$
\end{theorem} 

\begin{proof}[Proof.]
	
	Note that $R_{\partial\rightarrow id_0}=R_{id_0\rightarrow \int},$ where $\int$ is the operator defined by \eqref{int}. Recall that
	$$\int_0^z f(z)dz=S_{0,1} (h\ast f)(z), \quad h(z)=-\dfrac{\log(1-z)}{z}. $$
	Therefore, by Theorem 1, for $r<1$ and $a>r,$
	$$m_{id_0\rightarrow \int}(r,a)=ar+(1/a-a)r\Big(\dfrac{-\log(1-ar)}{ar}-1\Big).$$
	We equate the last expression to one and get 
	$$r(a)=\dfrac{1}{a}\Big(1+\dfrac{a^2-1}{2a^2-1}W\Big(\dfrac{1-2a^2}{a^2-1}e^{-1}\Big)\Big).$$
	
	Let us find at what point $r(a)$ under the condition $a\ge r(a)$ reaches its minimum. To do this, we draw  plots of functions $r=r(a)$ and $r=a$:

	\begin{figure}[h!]
		\centerline{\includegraphics[width=0.5\linewidth, height=7cm]{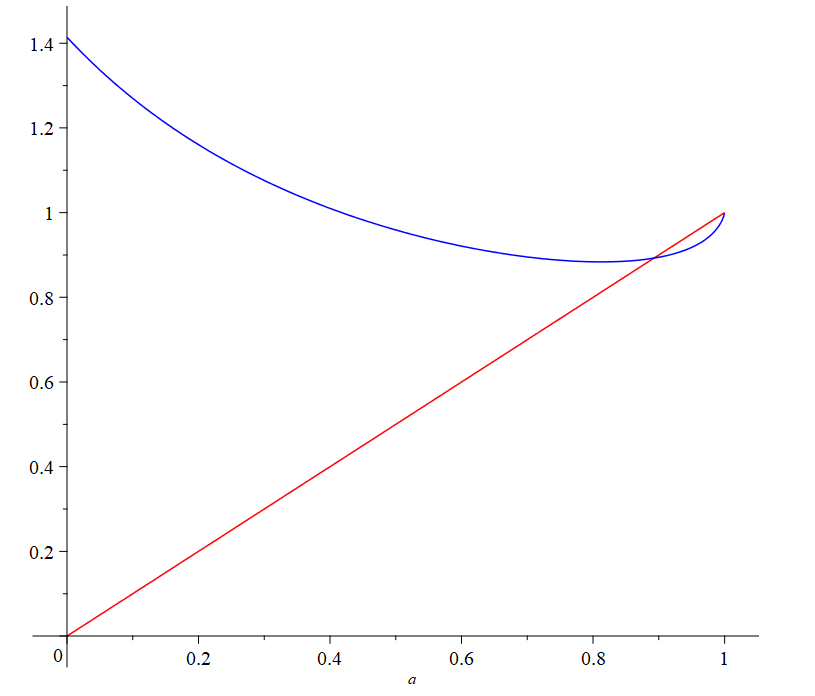}}  
	\end{figure}

	It can be seen  that under the condition $a \ge r(a)$ the function $r(a)$ takes the minimum value at the point $a$, such that $r(a)=a$. Solving the equation $a=r(a)$, we find that the method used is only suitable for $a>\Big({1+\dfrac{1}{2}W\Big(\dfrac{-2}{e^2 }\Big)}\Big)^{1/2}=0.892643...$
\end{proof}

\vspace{2mm}
4. R. Ali, R. Barnard and A. Solynin considered \cite{ABS} the problem of calculating the Bohr radius for the functions of the form $\sum_{n\ge 0}a_{nm}z^{nm}.$ In our terms their result can be written as follows: $R_{id \rightarrow A_{\frac{1}{1-z^m}}}=\dfrac{1}{\sqrt[m]{3}}.$
In \cite{KayumpLacun} I. Kayumov and S. Ponnusamy solved the same problem for functions of the form $z^l\sum_{n\ge 0}a_{nm}z^{nm}.$ In particular, from their proof one can obtain the values of $R_{id \rightarrow A_{\frac{1}{1-z^m}}}(a)$ for some $a\in [0,1).$ Let us show that these values can be calculated, using Theorem 1. Note that
$$R_{id \rightarrow A_{\frac{1}{1-z^m}}}=R_{id\rightarrow A^{m+1,-m-1}_{\frac{z^{m+1}}{1-z^m}}}.$$
Let 
$$h_1(z)=\dfrac{z^{m+1}}{1-z} \hspace{5mm}\text{and} \hspace{5mm} h_2(z)=\dfrac{z^{m+1}}{1-z^m}.$$
Then $$\widetilde{h_2}(z)=\dfrac{z}{1-z^m}=\dfrac{1}{m}\sum_{n=0}^{m-1}\dfrac{z}{1-e^{\frac{2\pi i n}{m}}z}\in {co}\mathcal{K}.$$ 
So we can use Theorem 1:
$$m_{id\rightarrow A^{m+1,-m-1}_{\frac{z^{m+1}}{1-z^m}}}(r,a)=a+(1/a-a)\dfrac{(ar)^m}{1-(ar)^m}.$$
Therefore $$R_{id \rightarrow A_{\frac{1}{1-z^m}}}(a)=\dfrac{1}{a}\sqrt[m]{\dfrac{a}{1+2a}}$$
for $a\in [0,1)$ such that $a^{2m-1}>\dfrac{1}{1+2a}.$

\vspace{3mm}
\section{The Bohr radius and the convergence radius of the convolution function}
It is easy to prove the following curious proposition, which we will use in what follows:

	\begin{theorem}\label{Th:6}
	Let $R_c(h)$ be the radius of convergence of $h$. Then 	\begin{equation}\label{conv}
		R_{A^{m,l}_h\rightarrow id} \cdot R_c(h)\le 1.	
	\end{equation}
\end{theorem} 

\begin{proof}[Proof.]
	Let ${k_s}$ be a sequence of natural numbers for which ${\sqrt[k_s]{|c_{k_s}}|}$ has the largest limit of all partial limits of the sequence ${\sqrt[k]{|c_k|}}$. Let us define a sequence of functions as follows:
	$$f_{k_s}(z)=\dfrac{z^{k_s}}{c_{k_s}}.$$
	Then $A^{m,l}_hf_{k_s}(z)=z^{k_s+l},$ hence $\|A^{m,l}_hf_{k_s}\|_{\infty}=1,$ but $M_rf_{k_s}= \dfrac{r^{k_s}}{|c_{k_s}|}.$ The last expression is greater than one if $r>\sqrt[k_s]{|c_{k_s}|},$ so
	$$R_{A^{m,l}_h\rightarrow id}\le\sqrt[k_s]{|c_{k_s}|}.$$
	Letting $s$ tend to infinity, we obtain the inequality $R_{A^{m,l}_h\rightarrow id}\le R^{-1}_c(h).$ 
\end{proof}

\vspace{3mm}
Using Theorem 6, we prove the following propositions:
	\begin{theorem}\label{Th:7}

	1. $R_{S_{m,-m} \rightarrow id} \rightarrow 1, \: m\rightarrow \infty;$
	\hspace{5mm}
	2. $R_{\partial^m \rightarrow id} \rightarrow 1, \: m\rightarrow \infty.$
\end{theorem}

\newtheorem*{C1}{Corollary}
\begin{C1}
	The number 1 in the  inequality \eqref{conv} can not be improved.
\end{C1}

\begin{proof}[Proof.]

	1. Let us find the lower estimate for $R_{S_{m,-m}\rightarrow id}.$ Note that $R_{S_{m,-m}\rightarrow id}=R_{id\rightarrow S_{m,m}}$ Suppose that $\|f\|_{\infty}\le 1.$ Then
$$	r^mM_rf=r^m\sum_{n\ge m}|a_n|r^n\le r^m \Big({\sum_{n\ge m}|a_n|^2}\Big)^{1/2}\Big(\sum_{n\ge m}r^{2n}\Big)^{1/2} =r^{2m}\dfrac{\|f\|_2}{\sqrt{1-r^2}}\le r^{2m} \dfrac{\|f\|_\infty}{\sqrt{1-r^2}}.$$
	The last expression is not greater than one if $r\le r_m$, where $r_m$ is the solution of the equation $r^{4m}+r^2=1.$ Obviously, $r_m\rightarrow 1, \:m\rightarrow\infty.$ It is clear that $R_{S_{m,-m}\rightarrow id}\le 1,$ therefore $R_{S_{m,-m}\rightarrow id} \rightarrow 1, \:m\rightarrow \infty.$
	
	\vspace{2mm}
	
	2. Obviously, $ R_{S_{m,-m}\rightarrow id} \le R_{\partial^m \rightarrow id},$ therefore the second statement of the Theorem follows from the first statement and also from Theorem 6.
\end{proof}

\vspace{1mm}
\section{Bohr radius and subordination of functions}
In \cite{BhowDas} B. Bhowmik and N. Das applied the Bohr theorem to questions on comparison of majorant series of subordinate functions. The definition of subordinate functions was given in section 4. Let us formulate the definition of quasisubordinate functions:

\newtheorem*{04}{Definition (M. Robertson, \cite{Rob})}
\begin{04}
	We say that the function $f(z)$ is quasisubordinate to the function $g(z)$ and write $f\prec_q g$ if there exists an analytic function $\omega(z)$ such that $|\omega(z)|\le 1, z\in \mathbb{D},$ $\omega(0)=0$ and
	$$|f(z)|\le |g(\omega(z))|, \quad z\in \mathbb{D}.$$
	
\end{04} 

The following theorem was proved in \cite{BhowDas} by B. Bhowmik and N. Das for the case of subordinate functions and generalized by I. Kayumov, S. Ponnusamy and S. Alkhaleefah  \cite{AKP} to the case of quasisubordinate functions:

\newtheorem*{E}{Theorem E}
\begin{E}
	Let  $f \prec_q g$ in $\mathbb{D}$. Then 
	$$M_rf \le M_rg, \quad 0\le r\le 1/3.$$
	The number $1/3$ is the best possible.
\end{E}

\vspace{1mm}
We generalize Theorem E for the cases of subordination and majorization.

	\begin{theorem}\label{Th:8}
	Let $m\ge 0, \: f(z)=\sum_{n\ge m}a_nz^n,\: g(z)=\sum_{n\ge m}b_nz^n$ and operator $T=A_h^{m,-m}$ defined by \eqref{Amh},  has the property $|c_n|\le |c_{n+1}|, \: n\ge m.$
	If $Tf \prec Tg$ in $\mathbb{D}$ then 
	$$M_rf\le M_r(Tg), \quad r\le R_{T\rightarrow id}.$$
\end{theorem} 

\vspace{1mm}
	\begin{theorem}\label{Th:9}
	Under the conditions of Theorem 8, if $|Tf(z)| \le |Tg(z)|,$ $z\in \mathbb{D},$ then
	$$M_rf\le M_r(Tg), \quad r\le R_{T\rightarrow id}.$$
	The number $r=R_{T\rightarrow id}$ in the last inequality is best  possible.
	
\end{theorem}
\vspace{1mm}
\begin{proof}[Proof of Theorem  8]
	Let $R_c$ be the radius of convergence of the power series $\sum_{n\ge m} c_nz^n.$ Since $|c_n|\le |c_{n+1}|,$ then $R_c\le 1. $ Note that
	$$T: \mathcal{H}_m(R_c^{-1}\mathbb{D}) \rightarrow \mathcal{H}(\mathbb{D}).$$
	Let $T^{-1}$ be the inverse operator for $T$, i.e.
	$$T^{-1}\Big(\sum_{n\ge 0}a_nz^n\Big)=\sum_{n\ge 0}\dfrac{a_n}{c_{n+m}}z^{ n+m}.$$
	Since $Tf\prec Tg$, there exists a function $\omega(z)$ analytic in $\mathbb{D}$ such that $|\omega(z)|\le 1, z\in \mathbb{D }, \:\omega(0)=0$ and
	$$Tf(z)=(Tg)(\omega(z)), \quad z\in \mathbb{D}.$$

	It follows from the Schwartz lemma that
	$$\Big|\dfrac{\omega^n(z)}{z^n}\Big|=\Big|T\circ T^{-1}\Big(\dfrac{\omega^n(z) }{z^n}\Big)\Big|\le 1, \quad z\in \mathbb{D}.$$
	That is why
	$$M_r\Big(T^{-1}\dfrac{\omega^n(z)}{z^n}\Big)\le 1, \quad r\le \min\{R_{T\rightarrow id }, R_c^{-1}\}.$$
	
	\vspace{3mm}
	It follows from Theorem 6 that $\min\{R_{T\rightarrow id}, R_c^{-1}\}=R_{T\rightarrow id}.$
	Let $\omega^n(z)=\sum_{l\ge n}\alpha_l^{(n)}z^{l}.$ Thus,
	$$T^{-1}\dfrac{\omega^n(z)}{z^n}=T^{-1}\sum_{l\ge n}\alpha_l^{(n)}z^{ l-n}=\sum_{l\ge n}\dfrac{\alpha_l^{(n)}}{c_{l-n+m}}z^{l-n+m}, \quad z\in R_{T\rightarrow id}\mathbb{D}.$$
	Replacing $n$ with $n-m$,
	\begin{equation} \label{substit}
		\sum_{l\ge n-m}\Big|\dfrac{\alpha_l^{(n-m)}}{c_{l-n+2m}}\Big|r^{l+m}\le r^{n-m} , \quad r\le R_{T\rightarrow id}.
	\end{equation}
	Consider the function $Tf(z)=(Tg)(\omega(z)):$
	$$
	(Tg)(\omega(z))=\sum_{n\ge m}c_nb_n\omega^{n-m}(z) =\sum_{n\ge m}c_nb_n\sum_{l\ge n-m}\alpha_l^ {(n-m)}z^l.
	$$
	Hence
	$$
	f(z)=T^{-1}Tf(z)=T^{-1}((Tg)(\omega(z))) =\sum_{n\ge m}c_nb_n\sum_{l\ge n-m}\alpha_l^{(n-m)}\dfrac{z^{l+m}}{c_{l+m}}.
	$$
	
	Finally, we use the condition $|c_n|\le |c_{n+1}| $ and inequality \eqref{substit}:
	\begin{multline*}
		M_rf\le  \sum_{n\ge m}|c_nb_n|\sum_{l\ge n-m}|\alpha_l^{(n-m)}|\dfrac{r^{l+m}}{|c_{l+m}|}\le \\ \le  \sum_{n\ge m}|c_nb_n|\sum_{l\ge n-m}|\alpha_l^{(n-m)}|\dfrac{r^{l+m}}{|c_{l-n+2m}|}\le\sum_{n\ge m}|c_nb_n|r^{n-m}=M_r(Tg).  
	\end{multline*}
\end{proof} 

\vspace{1mm}
\begin{proof}[Proof of Theorem 9]
	
	Since $|Tf(z)|\le |Tg(z)|, \: z\in \mathbb{D}$, there exists a function $\phi$ analytic in $\mathbb{D}$ such that $ |\phi(z)|\le 1, \:z\in \mathbb{D},$ and $$Tf(z)=\phi(z)Tg(z).$$
	Let $\phi(z)=\sum_{k\ge 0} \phi_kz^k.$ Then
	$$|\phi(z)|=|T\circ T^{-1}(\phi(z))|\le 1, \quad z\in \mathbb{D}.$$
	Therefore
	$$M_rT^{-1}\phi=\sum_{n\ge 0} \Big|\dfrac{\phi_n}{c_{n+m}}\Big|r^{n+m}=\sum_{ n\ge s} \Big|\dfrac{\phi_{n-s}}{c_{n-s+m}}\Big|r^{n-s+m}\le 1$$ for all $r\le \min\{R_{T\rightarrow id}, R_c^{-1}\}=R_{T\rightarrow id}.$
	Then
	$$Tf(z)=\sum_{n\ge m}c_na_nz^{n-m}=\sum_{k\ge 0}\sum_{n\ge m}\phi_kc_nb_nz^{k+n-m}=\sum_{n \ge 0}\sum_{s=m}^{m+n}\phi_{n-s+m}c_sb_sz^n.$$
	Let us apply the operator $T^{-1}$ to the obtained expression:
	$$f(z)=\sum_{n\ge m}a_nz^n=\sum_{n\ge 0}\sum_{s=m}^{m+n}\dfrac{\phi_{n-s+ m}c_sb_s}{c_{n+m}}z^{n+m}=\sum_{n\ge m}\sum_{s=m}^{n}\dfrac{\phi_{n-s}c_sb_s}{ c_{n}}z^{n}.$$
	Equating the coefficients on the right and left sides of the last equality, we obtain
	$$a_n=\sum_{s=m}^{n}\dfrac{\phi_{n-s}c_sb_s}{c_{n}}.$$
	Thus,
	\begin{multline*}
		M_rf=\sum_{n\ge m}|a_n|r^n\le \sum_{n\ge m}\sum_{s=m}^{n}\Big|\dfrac{\phi_{n-s}c_sb_s}{c_{n}}\Big|r^n\le\\ \le\sum_{s\ge m}\sum_{n\ge s}\Big|\dfrac{\phi_{n-s}}{c_{n-s+m}}\Big|r^{n-s+m}|c_sb_s|r^{s-m}\le \sum_{s\ge m}|c_sb_s|r^{s-m}=M_rTg.
	\end{multline*}

	The sharpness of $r=R_{T\rightarrow id}$ is obvious if we consider $g(z)=\dfrac{z^m}{c_m}$ and $f(z)$ and take as f an extremal function for the corresponding Bohr's problem (if there is no extremal function, we need to take an extremal sequence of functions). 
\end{proof}
\vspace{3mm}
\textbf{Acknowledgments.} The author thanks A.D. Baranov for his attention to the work on the article, B.N. Khabibullin for the question on the relation between subordination and the general concept of the Bohr radius, and V.V.~Shemyakov for a useful remark in the Bohr problem with a derivative.

The results of Sections 4 and 5 were obtained with the support of  the Russian Science Foundation projects № 23-11-00153. The results of Sections 6 and 7 were obtained with the support of  the Foundation for the Development of Theoretical Physics and Mathematics "BASIS".






\end{document}